\documentclass[11pt]{article}

\usepackage{amsmath,amssymb,amsthm,color}
\usepackage{mathtools}
\usepackage[shortlabels]{enumitem}
\usepackage{graphicx}
\usepackage{xcolor}

\usepackage[T1]{fontenc}
\usepackage{lmodern}
\usepackage{natbib}

\usepackage{hyperref}

\newtheorem{theorem}{Theorem}
\newtheorem{proposition}[theorem]{Proposition}

\newcommand{\sH}{\mathcal{H}}

\newcommand{\bbS}{\mathbb{S}}

\newcommand{\conv}{\mathrm{conv}}

\renewcommand{\Re}{\mathbb{R}}

\newcommand{\sB}{\mathcal{B}}

\newcommand{\td}{\widetilde{\delta}}

\newcommand{\deltaneg}{\delta_{\mbox{\footnotesize{neg}}}}

\newcommand{\tdelta}{\widetilde{\delta}}

\newcommand{\tentative}[1]{}
\DeclarePairedDelimiterX{\setnorm}[1]
  {\lvert\mkern-2mu\lvert\mkern-2mu\lvert}{\rvert\mkern-2mu\rvert\mkern-2mu\rvert}{#1}

\usepackage{fullpage} % Automatically reduces margins to fit more content

\title{Linear and Sublinear Diversities}

\author{David Bryant and Paul Tupper}

\begin{document}
\maketitle

\begin{abstract}
Diversities are an extension of the concept of a metric space which assign a non-negative value to every finite set of points, rather than just pairs. A general theory of diversities has been developed which exhibits many deep analogies to metric space theory but  also veers off in  new directions. Just as many of the most important aspects of metric space theory involve metrics defined on $\Re^k$, many applications of diversity theory require a specialized theory for diversities defined on $\Re^k$, as we develop here. We focus on two fundamental classes of diversities defined on $\Re^k$: those that are  Minkowski linear and those that are Minkowski sublinear. 
Many well-known functions in convex analysis belong to these classes, including  diameter, circumradius and  mean width. We derive surprising characterizations of these classes, and establish elegant connections between them. Motivated by classical results in metric geometry, and connections with combinatorial optimization, we then examine embeddability of finite diversities into $\Re^k$. We prove that a finite diversity can be embedded into a linear diversity exactly when it is of negative type and that it can be embedded into a sublinear diversity exactly when it corresponds to a generalized circumradius.
\end{abstract}

\section{Introduction}

We start with a brief introduction to the theory of diversities, state our main results, and then give an overview of the paper.

A diversity \cite{BryantTupper12} is a pair $(X,\delta)$ where $X$ is a set and $\delta$ is a non-negative function defined on finite subsets of $X$ satisfying
\begin{quotation}
\noindent  (D1) $\delta(A)\geq 0$ and $\delta(A)=0$ if and only if $|A|\leq 1$,\\
(D2) $\delta(A \cup C) \leq \delta(A \cup B) + \delta(B \cup C)$ whenever $B \neq \emptyset$.
\end{quotation}
\noindent As such, diversities are  set-based analogues of metric spaces. 
In fact, for a diversity $(X,\delta)$, if we let $d(x,y)=\delta(\{x,y\})$, then $(X,d)$ is a metric space, called the {\em induced metric space} \cite{BryantTupper12}.

Properties (D1) and (D2) are equivalent \cite{BryantNiesEtal21} to (D1) together with monotonicity
\begin{quotation}\noindent (D3) $\delta(A) \leq \delta(B)$ whenever $A \subseteq B$ \end{quotation}
and subadditivity on intersecting sets
\begin{quotation}
\noindent (D4) $\delta(A \cup B) \leq \delta(A)+ \delta(B)$ when $A \cap B \neq \emptyset$.
\end{quotation}
\noindent We say $(X,\delta)$ is a {\em semidiversity} if (D1) is relaxed to 
\begin{quotation}\noindent  (D1$'$) $\delta(A)\geq 0$ and $\delta(A)=0$ if $|A|\leq 1$.\end{quotation}
That is, sets with two or more points may have zero diversity. This terminology is analogous to at least some of the definitions of semimetrics. 

Many well-known set functions are diversities: for any metric space the diameter of a set, the length of a connecting Steiner tree \cite{BryantTupper12}, and the length of a minimal traveling salesperson tour \cite{BryantTupper14} are diversities; within $\Re^k$ the circumradius \cite{bryant2023diversities}, the mean width \cite{BryantTupper14}, and the length of a smallest enclosing zonotope (see Proposition 2) are diversities. Two well-known set functions which fail to be diversities are 
{\em average distance}
\[
\pi(A)
  = \frac{1}{\binom{|A|}{2}}
    \sum_{\{a,b\}\subset A} d(a,b).
\]
which is not monotonic (D3), and volume of convex hull, which fails (D4).

There are broad classes of diversities just like there are broad classes of metrics. The $\ell_1$ metrics on $\Re^k$
have the form 
\[d_1(a,b) = \sum_{i=1}^k |a_i - b_i|,\]
while {\em $\ell_1$ diversities} on $\Re^k$ \cite{BryantTupper14} have the form
\[\delta_1(A) = \sum_{i=1}^k \max_{a,b \in A} |a_i - b_i|.\]

For a finite set $X$, the class of (semi)-metrics of  negative type  is defined by the condition that  
$\sum_{a,b} x_a x_b \,d(a,b) \leq 0$
for all zero-sum vectors $x$. It includes $\ell_1$ metrics as a special case. In a similar way, for a finite set $X$, the class of semidiversities of negative type is defined by the condition that
\begin{equation}
     \sum_{A,B \subseteq X} x_A x_B\, \delta(A \cup B) \leq 0  \label{eq:negtype}
\end{equation}
for all zero-sum vectors $x$ with $x_\emptyset=0$. It includes $\ell_1$ diversities as a special case.  Metrics of negative type and diversities of negative type have multiple  characterizations \cite{DezaLaurent97,WuBryantEtal19}. We note that \cite{WuBryantEtal19} only consider {\em diversities} of negative type, whereas we also consider {\em semidiversities}. We shall see (Proposition~\ref{prop:negtype_properties}) that the conversion from one to the other is fairly automatic.

A {\em Minkowski diversity} $(\Re^k,\delta_K)$ is determined by a compact convex set $K \subseteq \Re^k$ with non-empty interior, known as the kernel \cite{bryant2023diversities}. The diversity $\delta_K(A)$ of a finite subset $A \subseteq \Re^k$ is the minimum amount that $K$ needs to be scaled so that a translate of $K$ covers $A$, that is, the generalized circumradius of $A$ with respect to $K$. \\

Our focus here is on the intersection of diversity theory, geometry and convex analysis. 
Recall that the Minkowski sum of two subsets $A, B \subseteq \Re^k$ is given by
\(
A+ B = \{ a + b : a \in A, b \in B\}
\)
and the scalar multiple of a set $A$ by $\lambda \in \Re$ is given by $\lambda A = \{ \lambda a : a \in A\}$.
For $b \in \Re^k$ we write  $A+b$ as a shorthand for $A + \{b\}$. A semidiversity $(\Re^k,\delta)$ is {\em Minkowski sublinear} if it satisfies 
\begin{quotation}
\noindent (D5) $\delta(\lambda A)  = \lambda \delta(A) \mbox{ and } \delta(A+B)  \leq \delta(A) + \delta(B)$
\end{quotation}
for $\lambda \geq  0$ and all nonempty finite subsets $A,B \subseteq \Re^k$. It is {\em Minkowski linear} if it also satisfies 
\begin{quotation}
\noindent (D6) $ \delta(A+B) = \delta(A) + \delta(B)$
\end{quotation}
for all nonempty finite subsets  $A,B \subseteq \Re^k$.
Many  familiar diversities defined on $\Re^k$ are Minkowski linear or sublinear (see below). We explore their properties and characterization. From here on we will use the terminology linear or sublinear as shorthand for Minkowski linear or sublinear respectively.

As per usual, we make repeated use of support functions when dealing with convex bodies and functions defined on them. The {\em support function}  of a nonempty bounded set $A$ is defined 
\[h_A:\Re^k \to \Re:x \mapsto \sup\{a\cdot x:a \in A\}.\]
Here $a \cdot x$ denotes  
the usual scalar product in $\Re^k$,
\[a \cdot x =\sum_{i=1}^k a_i x_i.\]
We make use of the following  properties of the support function, see \cite[Chapter 1]{Schneider14} for further details. 
    \begin{enumerate}
        \item A  nonempty bounded set has the same support function as both its closure and convex hull.
        \item $h_{A+B} = h_A + h_B$ and $h_{\lambda A} = \lambda h_A$ for non-empty, bounded $A,B$ and $\lambda \geq 0.$
        \item If $A,B$ are nonempty convex compact sets then $A \subseteq B$ if and only if $h_A(x) \leq h_B(x)$ for all $x \in \Re^k$.
     \end{enumerate}

    We often consider support functions restricted to $\bbS^{k-1}$, the unit sphere in $\Re^k$, noting that a support function is determined everywhere by its values on $\bbS^{k-1}$.
We note that the support function restricted to  $\bbS^{k-1}$ is bounded if and only if the set is bounded.

Our main results for diversities and semidiversities $(\Re^k,\delta)$ are:
\begin{enumerate}
\item (Theorem~\ref{thm:linear_characterize}) Linear diversities and semidiversities are exactly those which can be written in the form
\[ \delta(A) = \int_{\bbS^{k-1}} h_A(x) \mathrm{d}\nu(x) \]
for a finite positive Borel measure $\nu$ on the sphere $\bbS^{k-1}$ satisfying 
\begin{equation} 
    \int_{\bbS^{k-1}} x \, \mathrm{d}\nu(x) = 0.
\end{equation} 
\item (Theorem~\ref{thm:sublinearSup}) A diversity or semidiversity is sublinear if and only if it is the maximum of linear semidiversities (just like a function is convex if and only if it is the maximum of affine functions). 
\end{enumerate}

We then shift to studying the embeddings of finite diversities (that is,  diversities $(X,\delta)$ with $X$ finite) into linear or sublinear diversities.   
We say that (semi)diversity $(X_1,\delta_1)$ can be embedded in (semi)diversity $(X_2,\delta_2)$ if there is a map $f \colon X_1 \rightarrow X_2$ such that $\delta_2(f(A)) = \delta_1(A)$ for all finite $A \subseteq X_1$. If the source $(X_1,\delta_1)$ is a diversity, then the embedding $f$ is automatically injective by $(D1)$, but in general we can embed semidiversities into diversities  with non-injective $f$.
Questions regarding embeddings and approximate embeddings of finite metrics in normed spaces are central to metric geometry and its applications. Consider, for example, Menger's characterizations of when a metric can be embedded in Euclidean space, or the extensive literature applying metric embeddings to combinatorial optimization (reviewed in \cite{DezaLaurent97,matousek2013lectures,indyk20178}). 

For finite diversities $(X,\delta)$ we show:
\begin{enumerate}
\item (Theorem~\ref{thm:linearEmbed}) A finite semidiversity can be embedded in a linear diversity if and only if it can be embedded in a Minkowski diversity with simplex kernel (see Section~\ref{subsec:examples_of_lin_and_sublin} below) if and only if it is of negative type.
\item (Theorem~\ref{thm:sublinearEmbed}) A finite semidiversity  can be embedded in a sublinear diversity if and only if it can be embedded in a Minkowski diversity 
(see Section~\ref{subsec:examples_of_lin_and_sublin} below) if and only if 
it is the maximum of a finite collection of  semidiversities of negative type  on $X$.
\end{enumerate}

That the characterization of linear embeddable diversities came out as tidy as Theorem~\ref{thm:linearEmbed} is surprising. Theorem~\ref{thm:linearEmbed} can, in many ways, be seen as the analogue of the celebrated and useful characterization of  metrics of negative type as squares of Euclidean metrics. However the relationship between the metric and diversity  characterizations is not as straightforward as it might first appear. The metric result is not  a special case of the diversity result, in fact there are many metrics of negative type which are not the induced metrics of a diversity of negative type \cite{WuBryantEtal19}. 

These results follow a general pattern which we have seen multiple times in the development of diversity theory. Sometimes the theory of diversities  runs  in parallel with that of metric spaces; other times it veers off in new directions. In the first paper on diversities \cite{BryantTupper12} the authors explored how concepts of hyperconvexity, injectivity and the tight span extended, and to an extent enriched, the analogous metric concepts.  
Diversities turn out to be an exemplary class of metric structures, exhibiting fascinating connections with model theory and Urysohn's universal space \cite{BryantNiesEtal17,BryantNiesEtal21,Hallback20,hallback2020automorphism}. Other directions that have been pursued are a diversity analogue of ultrametric and normed spaces \cite{HaghmaramNourouzi20,MehrabaniNourouzi20,haghmaram2022diversity} and new diversity-based results in fixed point theory \cite{espinola2014diversities,Piatek14}.

There are many reasons to believe that diversities are to hypergraphs what metrics are to graphs. The main contribution of  \cite{BryantTupper14} is to show that the `geometry of graphs' linking metric embeddings to approximation algorithms on graphs  \cite{LinialLondonEtal95}  has a parallel `geometry of hypergraphs' linking diversity embeddings to approximation algorithms on hypergraphs. Jozefiak and Shepherd \cite{jozefiak2023diversity} use this approach to obtain the best known approximation algorithms for several hypergraph optimization problems. 

The approximation bounds achieved by \cite{LinialLondonEtal95} have  been superseded by much tighter bound based on  approximate embeddings of metrics of negative type. These tighter bounds motivate the study of negative type diversities \cite{WuBryantEtal19}. A geometric approach, as with metrics, has the potential to expand the metric embedding toolkit from graphs to hypergraphs. However diversity theory is new, and the fundamentals of the geometry of diversities are still being worked out. We see the results here as a contribution towards that endeavor. \\

The structure of the paper is as follows.
In Section~\ref{sec:linear_and_sublinear} we establish the basic properties and characterizations of linear and sublinear diversities.
Subsection \ref{subsec:examples_of_lin_and_sublin} gives examples of several linear and sublinear diversities, including defining Minkowski diversity. Subsection~\ref{subsec:properties_of_linear_and_sublin} establishes
many basic properties of  sublinear semidiversities that we need later for our main results.
Subsection~\ref{subsec:linear_characterization} gives Theorem~\ref{thm:linear_characterize}, a complete characterization of linear semidiversities.   Subsection~\ref{subsec:char_sublinear} gives Theorem~\ref{thm:sublinearSup}, our characterization of sublinear semidiversities. In Section~\ref{sec:embedding} we switch to studying when finite diversities can be embedded in either linear or sublinear diversities. 
There we prove Theorem~\ref{thm:linearEmbed}, a characterization of when a semidiversity is linear embeddable, and Theorem~\ref{thm:sublinearEmbed}, a characterization of when a semidiversity is sublinear embeddable.

\section{Linear and sublinear diversities} \label{sec:linear_and_sublinear}

In this section we establish basic properties and characterizations for linear and sublinear diversities. 

\subsection{Examples of Linear and Sublinear Diversities} \label{subsec:examples_of_lin_and_sublin}

We start with examples of diversities which are linear or sublinear. Note that for all diversities $(X,\delta)$ we have $\delta(\emptyset) = 0$, even if that is not stated explicitly below in the definitions. 

\begin{enumerate}
\item    Let $\| \cdot \|$ be any norm on $\Re^k$. The {\em diameter diversity} (for this particular space) is given by
    \[\delta(A) = \max_{a,b \in A} \|a-b\|\]
for  finite $A \subseteq \Re^k$. The diameter diversity on normed spaces is sublinear \cite[p. 49]{Schneider14}.
\item The $\ell_1$ diversity $(\Re^k,\delta_1)$ is 
\[
\delta_1(A) = \sum_{i=1}^k \max_{a,b \in A} |a_i - b_i| 
\]
for  finite $A\subseteq \Re^k$ \cite{BryantTupper14}. For finite $A,B$ and $\lambda \geq 0$ we have (noting $\max_{a,b \in A} |a_i - b_i| =\max_{a,b \in A} (a_i - b_i) $ )
\begin{align*}
\delta_1(\lambda A+B) & = \sum_{i=1}^k \max\left\{ \left((\lambda a+b)_i - (\lambda a' + b')_i\right):   a,a' \in A, \, b,b' \in B \right\} \\
& = \sum_{i=1}^k  \max\{\lambda(a_i - a_i') + (b_i - b_i') :  a,a' \in A, \, b,b' \in B \} \\
& =  \sum_{i=1}^k \max\{\lambda(a_i - a_i') : a,a' \in A \} +  \sum_{i=1}^k \max \{ (b_i - b_i') :   b,b' \in B \}  \\
& =\lambda \delta_1(A) + \delta_1(B).
\end{align*}
so $(\Re^k,\delta_1)$ is a linear diversity.
\item The {\em circumradius diversity} of finite $A \subseteq \Re^k$ 
with respect to the unit ball $\sB$ is  defined to be
\[\delta_\sB(A) = \inf\{\lambda  \geq 0: A \subseteq \lambda \sB + x \mbox{ for some $x \in \Re^k$} \}.
\]
More generally, we define the {\em Minkowski diversity} $(\Re^k,\delta_K)$ with kernel $K$ (also known as the generalized circumradius) to be  
\begin{equation} \label{eq:minkowski_defn}
\delta_K(A) = \inf \{\lambda \geq 0 : A \subseteq  \lambda K +x \mbox{ for some } x \in \Re^k\},
\end{equation}
for finite $A \subseteq \Re^k$. We assume that $K$ is compact and convex with non-empty interior.  Minkowski diversities (including the circumradius) are sublinear \cite{bryant2023diversities} but are not, in general, linear. For example, consider the circumradius diversity  $(\Re^2,\delta_\sB)$. If $A = \{(0,0),(1,0)\}$ and $B = \{(0,0),(0,1)\}$ then $\delta_\sB(A) = \delta_\sB(B) = 1/2$ but $\delta_{\sB}(A+B) = \frac{1}{\sqrt{2}}$, so $\delta_\sB(A+B) < \delta_\sB(A) + \delta_\sB(B)$. 
\item The {\em mean width diversity} $(\Re^k,\delta_w)$ \cite{BryantTupper14} is 
\[\delta_w(A) = %\frac{2}{\omega_k}
\frac{2}{\alpha_k} \int_{\bbS^{k-1}} h_A(x) \, \mathrm{d} \nu(x) \]
where $\nu(x)$ is the (uniform) Haar measure on the sphere and %$\omega_k = \int_{\bbS^{k-1}}  \, \mathrm{d} \nu(x)$. 
$\alpha_k$ is chosen so that $\delta_w(\{a,b\}) = \|a-b\|_2$ for all $a,b \in \Re^k$.
%Equivalently,
The value $\delta_w(A)$ is proportional to the mean width of the convex hull of $A$.  Mean width diversities are linear \cite[p. 50]{Schneider14} and 
%$\delta_w(\{a,b\}) = \|a-b\|_2$ for all $a,b \in \Re^k$ so its 
their induced metric is the Euclidean metric. 
%In the Euclidean plane $\delta_w(A)$ equals half the perimeter of $\conv(A)$, by Cauchy's formula.

Let $w_A(x) = \max\{x \cdot (a-b):a,b\in A\}$ denote the width of $A$ in direction $x$, so $w_A(x) = h_{A-A}(x) = h_A(x) + h_A(-x)$. Then 
\[ \delta_w(A) = \frac{1}{\alpha_k} \int_{\bbS^{k-1}} w_A(x) \, \mathrm{d} \nu(x). \]
For $1 \leq p < \infty$ we define
\[ \delta^{(p)}_w(A) = \frac{1}{\alpha_k} \left[\int_{\bbS^{k-1}} |w_A(x)|^p \, \mathrm{d} \nu(x) \right]^{1/p}. \]
That this is a sublinear diversity follows from the Minkowski inequality. See \cite[Prop.~2.4]{HaghmaramNourouzi20}, or \cite[Prop.~10]{BryantCioica-LichtEtal21} in the case that $p=2$. 
\item A {\em zonotope} $Z$ is a Minkowski sum of line segments. By linearity, the mean width diversity of a zonotope equals the sum of lengths of the line segments it is formed from (see \cite{joos2023isoperimetric}). This sum is called the {\em length} $\ell(Z)$ of the zonotope, and is well-defined even though a zonotope can be expressed as the sum of multiple different collections of line segments. 

We define the {\em zonotope diversity} $(\Re^k,\delta_z)$, where $\delta_z(A)$ is the minimum length of a zonotope containing $A$, and show that it is a sublinear diversity in Proposition~\ref{prop:zonotope}. Algorithms for computing the length of the minimum enclosing zonotope can be found in \cite{guibas2003zonotopes}.
\end{enumerate}

\subsection{Properties of linear and sublinear diversities} \label{subsec:properties_of_linear_and_sublin}

We establish some basic properties of sublinear diversities (and hence of linear diversities). This includes  the continuous extension of sublinear diversities from finite sets to bounded sets. 

\begin{proposition} \label{prop:sublinearProperties}
    Let $\delta$ be a function on finite subsets of $\Re^k$ which satisfies \textnormal{(D1)}, monotonicity \textnormal{(D3)}  and sublinearity \textnormal{(D5)}. Then the following hold:
    \begin{enumerate}
        \item  $\delta$ is translation invariant: $\delta(A + x) = \delta(A)$ for all finite $A \subseteq \Re^k$ and $x \in \Re^k$.
                \item $(\Re^k,\delta)$ is a diversity.
        \item  If $\conv(A) = \conv(B)$ then $\delta(A) = \delta(B)$.
        \item If $A \subseteq \conv(B)$ then $\delta(A) \leq \delta(B)$.
        \item The map $N:\Re^k \rightarrow \Re$ given by $N(x) = \delta(\{0,x\})$ is a norm on $\Re^k$ such that $\delta(\{x,y\}) = N(x-y)$ for all $x,y \in \Re^k$.
        \item For all finite $A \subseteq \Re^k$ with $|A| > 2$ we have
        \begin{align*}\delta(A) \leq \mbox{ $\frac{|A|-1}{|A|(|A|-2)}$}\sum_{a \in A} \delta(A \setminus \{a\}). \end{align*}
    \end{enumerate}
    If $\delta$ satisfies (D1$'$) rather than (D1) then  1.~through 6.~still hold except that $(\Re^k,\delta)$ is a semidiversity and $N$ is a seminorm.  
\end{proposition}
\begin{proof}
\begin{enumerate}
\item 
By sublinearity (D5) and (D1), we have 
\begin{align*}
\delta(A+x) &\leq \delta(A) + \delta(\{x\}) = \delta(A), \mbox{ and } \\ 
\delta(A)& \leq \delta(A+x) + \delta(\{-x\}) = \delta(A+x).
\end{align*}
\item As $\delta$ is monotonic and $\delta(\emptyset) = 0$, $\delta$ is non-negative, and by part 1., $\delta$ is  translation invariant. We show that $(\Re^k,\delta)$ satisfies (D4). Suppose that $x \in A \cap B$. Then $0 \in (A-x) \cap (B-x)$ and so $(A-x) \cup (B-x) \subseteq (A-x)+(B-x)$ and
\begin{align*} \delta(A \cup B) &= \delta\Big( (A -x) \cup (B -x ) \Big) \\&\leq \delta \Big( (A -x) + (B -x ) \Big) \\ 
&\leq \delta(A-x) + \delta(B-x) \\ &= \delta(A) + \delta(B).\end{align*}
Hence $(\Re^k,\delta)$ satisfies (D1), (D3) and (D4). 
\item By part 2., $(X,\delta)$ is a sublinear diversity. Then by \cite[Prop.\ 2.2b]{bryant2023diversities}, if $A, B$ are finite sets with $\conv(A)=\conv(B)$ then $\delta(A)=\delta(B)$.
\item We have $\conv(A \cup B) = \conv(B)$ so by 3., $\delta(B) = \delta(A \cup B) \geq \delta(A)$.
\item By (D5) we have $N(x+y) = \delta(\{0,x+y\}) \leq \delta(\{0,x\}) + \delta(\{0,y\}) = N(x) +N(y)$. If $\lambda\ge 0$ then $N(\lambda x) = \delta(\{0,\lambda x\}) = \lambda \delta(\{0,x\}) = |\lambda| N(x)$, while if $\lambda<0$ we have 
\[N(\lambda x) = \delta(\{\lambda x,0\}) = \delta(\{0,-\lambda x\}) = |\lambda| N(x).\]
Also, since $N(x)=\delta(\{0,x\})$, $N(x)=0$  if and only if $\delta(\{x,0\}) = 0$ if and only if $x=0$.
\item  Our proof is based on that of \cite[Theorem 4.1]{BrandenbergKonig13}. First suppose that $A \subseteq \Re^k$ satisfies 
\begin{equation}
    \sum_{a \in A} a = 0. \label{eq:zerosum}
\end{equation} For each $a \in A$ we then have $-a = \sum_{a' \in A \setminus \{a\}} a'$ and so 
\[ \frac{1}{|A|-1} (-a) = \sum_{a' \in A \setminus \{a\}} \frac{a'}{|A|-1} \in \conv(A \setminus \{a\}).\]
As well, we have for all $a' \neq a$ that  $a \in A \setminus \{a'\} \subseteq \conv(A \setminus \{a'\})$. Therefore
\begin{align*}
    \frac{(|A|-2)|A|}{|A|-1} a & = (|A|-1)a - \frac{1}{|A|-1}a \\
    & \in \sum_{a' \in A \setminus \{a\}} \conv(A \setminus \{a'\}) + \conv(A \setminus \{a\}) \\
    & = \sum_{a' \in A} \conv(A \setminus \{a'\}).
\end{align*}
This holds for all $a \in A$, so
\[  A \subseteq \frac{|A|-1}{(|A|-2)|A|}\sum_{a' \in A}\conv(A \setminus \{a'\}) = \conv\left(\frac{|A|-1}{(|A|-2)|A|}\sum_{a' \in A}(A \setminus \{a'\}) \right)\]
where we have used the identity $\conv(X+Y) = \conv(X) + \conv(Y)$, see \cite[Theorem~1.1.2]{Schneider14}. By part 4.~of the Proposition,
\begin{equation}\delta(A) \leq \mbox{ $\frac{|A|-1}{|A|(|A|-2)}$}\sum_{a \in A} \delta(A \setminus \{a\}). \label{eq:prop1_5}\end{equation}

Now suppose that  $\sum_{a \in A} a \neq 0$. Define $a_0 =  \frac{1}{|A|} \sum_{a \in A} a$ and  $A_0 = A - a_0$. Then $\sum_{a \in A_0} a = 0$ and from \eqref{eq:prop1_5} and translation invariance, 
\begin{align*}
    \delta(A) & = \delta(A_0) \\
    & \leq \mbox{ $\frac{|A_0|-1}{|A_0|(|A_0|-2)}$}\sum_{a \in A_0} \delta(A_0 \setminus \{a\}) \\
    & = \mbox{ $\frac{|A|-1}{|A|(|A|-2)}$}\sum_{a \in A} \delta((A-a_0) \setminus \{a-a_0\}) \\
    & = \mbox{ $\frac{|A|-1}{|A|(|A|-2)}$}\sum_{a \in A} \delta(A \setminus \{a\})
\end{align*}
as required.
\end{enumerate}
\end{proof}

We now show that the {\em zonotope diversity} introduced in Section~\ref{subsec:examples_of_lin_and_sublin} is in fact a sublinear diversity. The length of a zonotope equals its mean width (with the appropriate scaling) so for any two zonotopes $Z_1,Z_2 \subseteq \Re^k$ and $\alpha,\beta \geq 0$ we have that $Z_1 \subseteq Z_2$ implies $\ell(Z_1) \leq \ell(Z_2)$ and 
\[ \ell(\alpha Z_1 + \beta Z_2) = \alpha \ell(Z_1) + \beta \ell(Z_2).\]

The zonotope diversity of finite $A \subseteq \Re^k$ is 
\[\delta_z(A) = \inf\{\ell(Z): A \subseteq Z, \mbox{ $Z$ a zonotope}\}.\]
The induced metric of the zonotope diversity is the Euclidean metric.

\begin{proposition} \label{prop:zonotope}
The zonotope diversity $(\Re^k,\delta_z)$ is a sublinear diversity.
\end{proposition}
\begin{proof}
Recall that $\delta_z(A)$ is the infimum of the lengths of a zonotopes containing $A$. The function $\delta_z(A)$ is clearly monotonic, vanishes when $|A| \leq 1$, and is strictly positive when $|A|>1$. 
Let $\epsilon >0$ be given. 
Given finite $A$ and $\lambda > 0$, let $Z_A$ be a zonotope containing $A$ with $\ell(Z_A)< \delta_z(A)+ \epsilon$. Then $\lambda Z_A$ (with length less than $\lambda \delta_z(A) + \lambda \epsilon$) contains $\lambda A$, so $\delta_z(\lambda A)\leq \lambda \delta_z(A) + \lambda \epsilon$. Letting $\epsilon$ go to zero gives $\delta_z(\lambda A) \leq \lambda \delta_z(A)$. The other direction ($\lambda \delta_z(A) \leq \delta_z(\lambda A))$ is obtained similarly and thus gives positive homogeneity. Given finite $A,B \subseteq \Re^k$ and $\epsilon>0$, let $Z_A$ and $Z_B$ be zonotopes containing $A$ and $B$ respectively, with $\ell(Z_A) \leq \delta_z(A) + \epsilon/2$ and $\ell(Z_B) \leq \delta_z(B) +  \epsilon/2$. Then  $Z_A + Z_B$ is a zonotope containing $A+B$ with length  $\ell(Z_A) + \ell(Z_B) \leq \delta_z(A) + \delta_z(B) + \epsilon$. Hence $\delta_z(A+B) \leq \delta_z(A) + \delta_z(B) +\epsilon$. Letting $\epsilon$ go to zero gives sublinearity. By Proposition~\ref{prop:sublinearProperties} part 2, $(\Re^k,\delta_z)$ is a sublinear diversity.  
\end{proof}

As an example, let $T$ be the three vertices of an equilateral triangle with side length $1$ and centroid at the origin (Fig.~\ref{fig:triangle}). Gr\"unbaum \cite[pg.~257]{grunbaum1963measures}, citing \cite{Kovetz1962SomeExtremalProblems}, states that a symmetric convex set containing the equilateral triangle has perimeter at least $2 \sqrt{3}$ and that this is bound realized by a hexagon with side length $\frac{1}{\sqrt{3}}$ (Fig.~\ref{fig:triangle}). As any symmetric convex polygon in $\Re^2$ is a zonotope, $\delta_z(T) = \delta_z(-T)=\sqrt{3}$.

\begin{figure}[ht]
\begin{center}
    \includegraphics[width=0.8\textwidth]{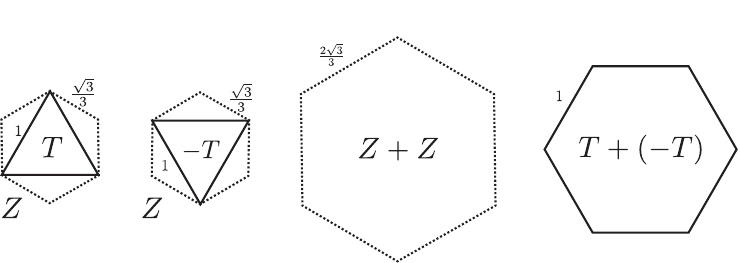}
\end{center}
\caption{\label{fig:triangle} An example illustrating that the zonotope diversity is not linear. $T$ is the set of vertices of an equilateral triangle, centred at the origin. $Z$ is the smallest zonotope containing $T$, which is the same as the smallest zonotope containing $-T$. The length of the convex hull of $T-T$ is $3$, which is smaller than $2 \sqrt{3}$, the length of $Z+Z$.  Hence $\delta_z(T-T) < \delta_z(T) + \delta_z(-T)$.}
\end{figure}

The convex hull of $T-T$ is the regular hexagon, centered at $0$, with side length $1$. This is already a zonotope, so $\delta_z(T - T) = 3$. We therefore have that $\delta_z(T + (-T)) < \delta_z(T) + \delta_z(-T)$, showing that  the zonotope diversity is not, in general, linear. \\

Let $\| \cdot \|$ be a norm on $\Re^k$ with associated metric $d(x,y) = \|x-y\|$ and unit ball $\sB = \{x : \|x\| \leq 1\}$. The {\em Hausdorff distance} between two nonempty closed bounded sets $K,L \subset \Re^k$ can be defined by \cite[p.\ 61]{Schneider14} : 
\[d_H(K,L) = \min \{\lambda: K \subseteq L + \lambda \sB \mbox{ and } L \subseteq K + \lambda \sB\}. \]
We extend the  function $\delta$ for a sublinear diversity $(X,\delta)$ from finite subsets to bounded $K \subseteq \Re^k$.  Define 
    \begin{equation} \delta^*(K) = \sup\{\delta(A) : A \subseteq K \mbox{ finite} \}. \label{eq:deltastar} \end{equation}

\begin{proposition} \label{prop:dstar}
Let $(\Re^k,\delta)$ be a sublinear semidiversity. 
\begin{enumerate}
\item For all bounded $K \subseteq \Re^k$, $\delta^*(K) < \infty$.
\item For all finite $A \subseteq \Re^k$,
\[ \delta^*(\conv(A)) = \delta(A). \]
\item For all bounded $K,L \subset \Re^k$ and $\lambda \geq 0$
\[ \delta^*(K+L) \leq \delta^*(K) + \delta^*(L) \]
and
\[ \delta^*(\lambda K) = \lambda \delta^*(K).\]
If $(\Re^k,\delta)$ is linear then
\[ \delta^*(K+L) = \delta^*(K) + \delta^*(L).\]
\item If $(\Re^k,\delta)$ is linear then the restriction of $\delta^*$ to the set of nonempty compact convex subsets of $\Re^k$ is a valuation. That is,
\[\delta^*(K \cap L) + \delta^*(K \cup L) = \delta^*(K) + \delta^*(L)\]
for all nonempty compact convex bodies $K,L$ such that $K \cap L$ and $K \cup L$ are non-empty and convex.
\item The restriction of $\delta^*$ to the set of nonempty compact convex subsets of $\Re^k$ is Lipschitz continuous with respect to the Hausdorff metric, with Lipschitz constant $\delta^*(\sB)$.
\end{enumerate}
\end{proposition}
\begin{proof}
\begin{enumerate}
\item
Suppose that $K \subseteq [-r,r]^k$ for some $r<\infty$. Let $e_1,\ldots,e_k$ be the standard basis for $\Re^k$ and define 
\[V = \sum_{i=1}^k \{-re_i,re_i\}\]
so $K \subseteq \conv(V) = [-r,r]^k$.
By Proposition~\ref{prop:sublinearProperties}, part 4., $\delta(A) \leq \delta(V)$ for all finite $A \subseteq K$ so 
\[\delta^*(K) \leq \delta(V) \leq \sum_{i=1}^k \delta(\{-re_i,re_i\}) < \infty.\] 
\item 
Let $A$ be a finite subset of $\Re^k$ and let $K = \conv(A)$. For any finite $A' \subseteq K$ we have $\conv(A \cup A') = \conv(A)$ so $\delta(A') \leq \delta(A \cup A') = \delta(A)$ by Proposition~\ref{prop:sublinearProperties} (3). 
Hence 
\[\delta(A) \leq \sup\{\delta(A'): \mbox{ finite } A' \subseteq K  \} \leq \delta(A).\]
\item 
Fix $\epsilon>0$ and suppose that $C$ is a finite subset of $K+L$ such that $\delta(C) > \delta^*(K+L) - \epsilon$. For each $c \in C$ there is $a_c \in K$ and $b_c \in L$ such that $c = a_c + b_c$. Let $A = \{a_c:c \in C\} \subseteq K$ and $B = \{b_c:c \in C\} \subseteq L$ so that $C \subseteq A+B$. It follows that 
\[\delta^*(K+L) - \epsilon < \delta(C) \leq \delta(A+B) \leq  \delta(A) + \delta(B) \leq \delta^*(K) + \delta^*(L).\]
Taking $\epsilon$ to zero gives the result. 

For any $K$ and $\alpha \geq 0$ we have
\[
\delta^*(\alpha K)= \sup_{\text{finite } A:\, A  \subseteq  \alpha K } \delta(A) = \sup_{\text{finite } B:\, \alpha B \subseteq  \alpha K } \delta(\alpha B) = \sup_{\text{finite } B:\, B \subseteq  K } \alpha \delta(B) = \alpha \delta^*(K)
\]
where we have used the change of variables $A=\alpha B$.

Suppose that $(\Re^k,\delta)$ is linear. Given $\epsilon>0$ there are finite $A \subseteq K$ and $B \subseteq L$ such  that 
\[\delta^*(K) < \delta(A) + \epsilon/2 \mbox{ and } \delta^*(L) < \delta(B) + \epsilon/2.\]
We then have 
\[\delta^*(K) + \delta^*(L) < \delta(A) + \delta(B) + \epsilon = \delta(A+B) + \epsilon \leq \delta^*(K+L) + \epsilon\]
as $A+B \subseteq K+L$. Taking $\epsilon \rightarrow 0$ and applying sublinearity gives
\[\delta^*(K + L) = \delta^*(K) + \delta^*(L).\]
\item By \cite[Lem.~3.1.1]{Schneider14} we have that if $K,L,K\cup L$ and $K\cap L$ are nonempty compact convex subsets then 
\[(K \cup L) + (K \cap L) = K+L. \]
By linearity, 
\[\delta^*(K \cup L) + \delta^*(K \cap L) = \delta^*(K) + \delta^*(L).\]
\item Suppose that $K,L$ are bounded nonempty subsets satisfying $d_H(K,L) = \lambda$. For any $\epsilon>0$ there is a finite $A \subseteq K$  such that $\delta(A) \leq \delta^*(K) < \delta(A) + \epsilon$. We also have $A \subseteq K \subseteq L + \lambda \sB$ so there are finite $B \subseteq L$ and $C \subseteq \sB$ such that $A \subseteq B + \lambda C$. Hence
\[ \delta^*(K) < \delta(A) + \epsilon \leq  \delta(B) + \lambda \delta(C) + \epsilon \leq \delta^*(L) + \lambda \delta^*(\sB) + \epsilon.\]
By a symmetric argument, 
\[ \delta^*(L) < \delta^*(K) + \lambda \delta^*(\sB) + \epsilon.\]
Taking $\epsilon$ to zero, we have
\[ |\delta^*(K) - \delta^*(L) | \leq \delta^*(\sB) d_H(K,L).\]

The bound is tight, as can be seen by letting $K = \sB$ and $L = 2 \sB$. Then $d_H(K,L) = 1$ and $|\delta^*(K) - \delta^*(L) | = \delta^*(\sB)$. 
\end{enumerate}
\end{proof}

Bryant et al.~\cite{bryant2023diversities} also describe  an extension of Minkowski diversities from finite sets to bounded sets. They define $\tdelta(P) = \delta(\mathrm{vert}(P))$ for any polytope with vertex set $\mathrm{vert}(P)$, and extend that to general bounded convex sets $K$ by defining $\tdelta(K) = \lim_{n \rightarrow \infty} \tdelta(P_n)$ for any sequence of polytopes $P_1,P_2,\ldots$ converging to $K$ under the Hausdorff metric. 

From Proposition~\ref{prop:dstar} part 5. we have that $\tdelta$ is well defined for any sublinear diversity. Proposition~\ref{prop:dstar} part 2. gives that $\delta^*(P) = \tdelta(P)$ for any polytope and from Proposition~\ref{prop:dstar} part 5 we have that $\delta^*(K) = \tdelta(K)$. Hence $\delta^*$ coincides with $\tdelta$ for sublinear diversities.

We conclude this section by showing that linear maps of linear (sublinear) semidiversities are linear (resp.~sublinear). 

\begin{proposition} \label{prop:linearMap}
    Let $\phi:\Re^m \rightarrow \Re^n$ be linear and let $(\Re^n,\delta_n)$ be a sublinear semidiversity. Then $(\Re^m,\delta_m)$ given by $\delta_m(A) = \delta_n(\phi(A))$ is a sublinear semidiversity. If $(\Re^n,\delta_n)$ is also linear then so is $(\Re^m,\delta_m)$.
\end{proposition}
\begin{proof}
    As defined, $\delta_m$ satisfies (D1$'$) and (D3). For finite $A,B \subseteq \Re^m$ and $\alpha,\beta \geq 0$ we have
    \[\delta_m(\alpha A + \beta B) = \delta_n\big(\phi(\alpha A + \beta B)\big) = \delta_n\big(\alpha \phi(A) + \beta \phi(B)\big).\]
    If $\delta_n$ is sublinear then we have
    \[\delta_n\big(\alpha \phi(A) + \beta \phi(B)\big) \leq \alpha \delta_n\big(\phi(A)\big) + \beta \delta_n\big(\phi(B)\big) = \alpha \delta_m(A) + \beta \delta_m(B) \]
so $(\Re^m,\delta_m)$ satisfies (D5) and is a sublinear diversity by Proposition~\ref{prop:sublinearProperties}. If $\delta_n$ is also linear then
 \[\delta_n\big(\alpha \phi(A) + \beta \phi(B)\big) = \alpha \delta_n\big(\phi(A)\big) + \beta \delta_n\big(\phi(B)\big) = \alpha \delta_m(A) + \beta \delta_m(B) \]
   and $(\Re^m,\delta_m)$ is linear.
\end{proof}

\subsection{Characterization of linear diversities} \label{subsec:linear_characterization}

The following characterization of linear diversities is essentially contained in the proof of the main theorem in Firey \cite{firey1976functional}. A clear statement of the relevant version of the Riesz representation theorem  may be found in \cite{rotem2021riesz}; see also \cite{meyer1995convex,hug2020lectures}.

\begin{theorem}\label{thm:linear_characterize}
Let $\delta$ be a function defined on finite subsets of $\Re^k$. Then $(\Re^k,\delta)$ is a linear semidiversity if and only if there is a positive finite Borel measure $\nu$ on the unit sphere $\bbS^{k-1} = \{x \in \Re^k: \|x\|_2 = 1\}$  such that 
\begin{equation} \label{eq:centroid}
    \int_{\bbS^{k-1}} x \, \mathrm{d}\nu(x) = 0
\end{equation} 
and
\begin{equation} 
\delta(A) = \int_{\bbS^{k-1}} h_A(x) \, \mathrm{d}\nu(x) \label{eq:linearInt} \end{equation}
for all finite $A \subseteq \Re^k$. Such a measure is unique.
\end{theorem}
\begin{proof}
First we show that $\delta$ given by \eqref{eq:linearInt} is a linear semidiversity. For $a \in \Re^k$
\begin{align*}
\delta(\{a\}) & = \int_{\bbS^{k-1}} h_{\{a\}}(x) \, \mathrm{d}\nu(x) 
 = \int_{\bbS^{k-1}} a \cdot x \,  \mathrm{d}\nu(x) 
 = a \cdot \int_{\bbS^{k-1}}  x \, \mathrm{d}\nu(x) 
 = 0.
\end{align*}
For finite $A,B \subseteq \Re^k$ and $\lambda \geq 0$, $h_{\lambda A}=\lambda h_A$ and $h_{A+B} = h_A+h_B$ so $\delta(\lambda A)= \lambda \delta(A)$ and  $\delta(A+B) = \delta(A) + \delta(B)$. If $A \subseteq B$ then  $h_A(x) \leq h_B(x)$ for all $x \in \Re^k$, giving $\delta(A) \leq \delta(B)$ since $\nu$ is positive. By Proposition~\ref{prop:sublinearProperties}, part 2., 
$(\Re^k,\delta)$ is a linear semidiversity.

For the converse, let $(\Re^k,\delta)$ be a linear semidiversity and define $\delta^*$ as in \eqref{eq:deltastar}. 
By Proposition~\ref{prop:dstar} the restriction of $\delta^*$ to nonempty compact convex subsets is Minkowski linear and monotonic. 
The Lipschitz property in Proposition~\ref{prop:dstar} part 5 means we can use \cite[Thm.\ 1.2]{rotem2021riesz}
to obtain that
there is a unique positive finite Borel measure $\nu$ such that 
\[
\delta^*(K) = \int_{\bbS^{k-1}} h_K(x) \, \mathrm{d}\nu (x)
\]
for all compact convex sets $K$.
Note that for all $a \in \Re^k$ we have
\[
0= \delta^*(\{a\}) = \int_{\bbS^{k-1}} a \cdot x \, \mathrm{d}\nu (x) = a \cdot  \int_{\bbS^{k-1}}  \ x \, \mathrm{d}\nu (x)
\]
and so $\int_{\bbS^{k-1}}  \ x \, \mathrm{d}\nu (x)=0$.
 Now for any nonempty finite $A$, let $K = \conv(A)$. Then $\delta(A) = \delta^*(K)$ by Proposition~\ref{prop:dstar}, part 2., and $h_K = h_A$, so
\[\delta(A) = \delta^*(K) = \int_{\bbS^{k-1}} h_K(x) \, \mathrm{d}\nu (x) = \int_{\bbS^{k-1}} h_A(x) \, \mathrm{d}\nu(x), \]
as required.
\end{proof}

\begin{figure}[ht]
\centering
\includegraphics[width=0.8\textwidth]{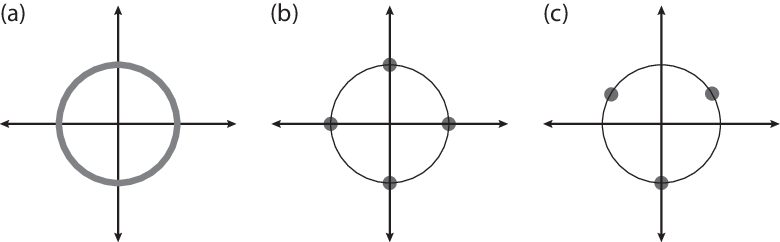}
\caption{\label{fig:linearMeasures} Support of measures corresponding to (a) mean width; (b) the $\ell_1$ diversity; and (c) a Minkowski diversity with a simplex kernel.
}
\end{figure}

In Figure~\ref{fig:linearMeasures} we depict the support for the measures corresponding to three different linear diversities:
\begin{enumerate}
    \item[(a)] The mean width diversity, where $\nu$ is the uniform Haar measure on the sphere.
    \item[(b)] The $\ell_1$ diversity, where $nu$ has point masses at $\pm e_i$, since 
\begin{align*}
    \delta(A) &= \int_{\bbS^{k-1}} h_A(x) \, \mathrm{d}\nu(x) 
     = \sum_{i=1}^k (h_A(e_i) + h_A(-e_i) )
     = \sum_{i=1}^k \max_{a,b \in A} |a_i - b_i|.
\end{align*}
 \item[(c)] The minimal case, in the sense that the support of $\nu$ is a set of affinely independent points, equivalently, the vertices of a simplex. No linear diversity has a smaller support. The corresponding diversities are characterized in the following Proposition.
 \end{enumerate}

\begin{proposition} \label{prop:linear_prog_simplex}
Let $v_0,\ldots,v_k \in \Re^k$ be affinely independent and suppose that $\sum_{i=0}^k c_i v_i =0$ for positive 
$c_0,\ldots,c_k$ such that $\sum_{i=0}^k c_i=1$. 
Define the polyhedron  $K = \bigcap_{i=0}^k \{y \in \Re^k:y \cdot v_i \leq 1 \}$. Let $\delta_K$ be the Minkowski diversity with kernel $K$.  
Then
\[
\delta_K(A) =  \sum_{i=0}^k c_i h_A(v_i)
\]
for all finite $A \subseteq \Re^k$. It follows that $(\Re^k,\delta_K)$ is a linear diversity.
\end{proposition}
\begin{proof}
Suppose that  $A = \{a_1,a_2,\ldots,a_m\}$. We express $\delta_K(A)$ as the solution to a linear program which is a special case of a more general formulation in \cite{BrandenbergRoth09}:
\begin{align*}
    \text{min}_{\lambda,x} \quad & \lambda\\
    \text{subject to} \quad 
    & \lambda + v_i \cdot x \geq v_i \cdot a_j \\ &\quad \mbox{ for all }i=0,1,\ldots,k \mbox{ and }j=1,\ldots,m.
\end{align*}
The dual linear program has  variables $y_{ij}$ for $i=0,1,\ldots,k$ and $j=1,\ldots,m$:
\begin{align*}
    \text{max}_y \quad & \sum_{i=0}^k \sum_{j=1}^m (v_i \cdot a_j) y_{ij} \\
    \text{subject to} \quad 
    & y_{ij} \geq 0, \mbox{ for all $i=0,\ldots,k$ and $j=1,\ldots,m$}\\
    & \sum_{i=0}^k \sum_{j=1}^m  y_{ij} = 1, \\
    & \sum_{i=0}^k \sum_{j=1}^m v_i y_{ij} = 0.
\end{align*}
Define $z_i = \sum_{j=1}^m y_{ij} \geq 0$ for all $i=0,\ldots,k$. Then our dual constraints are equivalent to 
\[
\sum_{i=0}^k z_i = 1, \ \ \ \sum_{i=0}^k z_i v_i = 0.
\]
Since the $v_\ell$ are affinely independent and  $\sum_{i=0}^k c_i v_i=0$, $\sum_{i=0}^k c_i=1$ we have $z_i=c_i$ for all $i=0,\ldots,k$ when $y$ is feasible.

For each $i=0,\ldots,k$ the maximum of $\sum_{j=1}^m (v_i \cdot a_j) y_{ij}$ such that $y_{i1},\ldots,y_{im}$ are non-negative and $z_i=\sum_{j=1}^m y_{ij} = c_i$ is obtained when $y_{ij^*} = c_i$ for $j^*$ maximizing $(v_i \cdot a_j)$ and $y_{ij}=0$ for $j \neq j^*$. This maximization can be carried out independently for each $i=0,\ldots,k$. Hence the optimal value for the dual and primal problems is 
\[
 \sum_{i=0}^k c_i \max_{a_j \in A} v_i \cdot a_j =  \sum_{i=0}^k c_i h_A(v_i),
\]
as required.
\end{proof}

\subsection{Characterization of sublinear diversities} \label{subsec:char_sublinear}

We now turn our attention to sublinear diversities. We will show that the relationship between sublinear and linear diversities parallels that between convex and affine functions. Just as every convex function is the supremum of affine functions, every sublinear diversity is the supremum of linear diversities (Theorem~\ref{thm:sublinearSup}). In fact, in our case, the supremum is attained for each set,  so the value of every sublinear diversity on a set is the maximum of the value of  a family of linear diversities on the set. Our proof 
relies heavily on the `Sandwich Theorem' (Theorem 1.2.5) of \cite{FuchssteinerLusky81convex}.

\begin{theorem} \label{thm:sublinearSup}
    Let $\delta$ be a function on finite subsets of $\Re^k$. 
    If $(\Re^k,\delta)$ is a sublinear diversity or semidiversity then there is a collection $\{(\Re^k,\delta_\gamma)\}_{\gamma \in \Gamma}$ of linear semidiversities such that 
    \[\delta(A) = \max\{\delta_\gamma(A):\gamma \in \Gamma\}.\]
Conversely, for any collection $\{(\Re^k,\delta_\gamma)\}_{\gamma \in \Gamma}$ of linear semidiversities and $\delta$ defined by $\delta(A) = \sup_{\gamma \in \Gamma} \delta_\gamma(A)$, the pair $(\Re^k, \delta)$ is a sublinear semidiversity.
\end{theorem}
\begin{proof}
Suppose that $\{(\Re^k,\delta_\gamma)\}_{\gamma \in \Gamma}$ are linear semidiversities and 
 \[\delta(A) = \sup\{\delta_\gamma(A):\gamma \in \Gamma\}\]
    for all finite $A \subseteq \Re^k$. 
Note that $\delta$ vanishes on singletons and is monotonic since each $\delta_\gamma$ has these properties. 
Suppose that $A,B$ are finite subsets of $\Re^k$ and $\lambda \geq 0$. Then 
\[ \delta(\lambda A) = \sup\{\delta_\gamma(\lambda A):\gamma \in \Gamma\} = \sup\{\lambda \delta_\gamma(A):\gamma \in \Gamma\} = \lambda \delta(A)\]
and
\[\delta(A+B) = \sup\{\delta_\gamma(A+B):\gamma \in \Gamma\} = \sup\{\delta_\gamma(A) + \delta_\gamma(B):\gamma \in \Gamma\} \leq \delta(A) + \delta(B),\]
so $(\Re^k,\delta)$ is sublinear.
By Proposition~\ref{prop:sublinearProperties}, $(\Re^k,\delta)$ is a sublinear semidiversity.

For the converse, suppose that $(\Re^k,\delta)$ is sublinear. Define $\sH$ to be the set of all support functions $h_A$ for nonempty finite $A \subseteq \Re^k$. Define $p$ on the convex cone $\sH$ by $p(h_A) = \delta(A)$ for all finite sets $A$. By Proposition~\ref{prop:sublinearProperties} part 3, $\delta(A)$ only depends on $\conv(A)$ so $p$ is well-defined on $\sH$. The function $p$ is sublinear as
for any finite $A,B$,
\[
p(h_A + h_B)= p(h_{A+B}) = \delta(A + B) \leq \delta(A) + \delta(B) = p(h_A)+p(h_B),
\]
and $p(\lambda h_A)=p(h_{\lambda A}) = \delta(\lambda A) = \lambda \delta(A) =\lambda p(h_A)$ for $\lambda \geq 0$.

Fix finite $B \subseteq \Re^k$. 
Define $q_B$ on $\sH$ by 
\[q_B(h_A) = \sup\{\lambda: \lambda B +x \subseteq \conv(A) \mbox{ for some $x \in \Re^k$}\}\]
with $q_B(h_A) = 0$ if $|A|=1$. When $|A| \geq 2$,  $q_B(h_A)$ is the largest we can scale $B$ so that a translate is contained in $\conv(A)$.
Note that $q_B(h_B)=1$ and hence $p(h_B)=\delta(B) = \delta(B) q_B(h_B)$.

We show that $q_B$ is superlinear. For all $\alpha \geq 0$ we have $q_B(\alpha h_A) = q_B(h_{\alpha A}) = \alpha q_B(h_A)$. 

Let $A_1,A_2$ be finite, non-empty subsets of $\Re^k$. If $|A_1|=1$ then $A_1+A_2$ is a translate of $A_1$, so 
\[q_B(A_1+A_2) = q_B(A_2) = q_B(A_1) + q_B(A_2).\]
Likewise, if $|A_2|=1$. 

Suppose that $|A_1|>1$ and $|A_2|>1$.
Given $\epsilon>0$ there are $\lambda_1 > q_B(h_{A_1}) - \epsilon/2$,  $\lambda_2 > q_B(h_{A_2}) - \epsilon/2$, $x_1,x_2 \in \Re^k$ such that 
\begin{align*}
    \lambda_1 B + x_1 &\subseteq \conv(A_1) \\
    \lambda_2 B + x_2 &\subseteq \conv(A_2) 
    \intertext{ and hence} 
    (\lambda_1 + \lambda_2) B + (x_1 + x_2) & \subseteq \conv(A_1) + \conv(A_2) \\
    & = \conv(A_1 + A_2),
\end{align*}
so that $q_B(h_{A_1 + A_2}) \geq (\lambda_1 + \lambda_2) > q_B(h_{A_1}) + q_B(h_{A_2}) - \epsilon$. Taking $\epsilon \rightarrow 0$ gives superlinearity. 

We now have that $p$ is monotonic and sublinear and that $q_B$ is superlinear. Again we have that for any finite $A \subseteq \Re^k$ and $\epsilon > 0$ there is $\lambda$ such that $q_B(h_A) - \epsilon < \lambda \leq q_B(h_A)$ and $x \in \Re^k$ with 
$\lambda B + x \subseteq \conv (A)$, and so 
\begin{align*}
    p(h_A) & \geq p(h_{\lambda B})
         = \lambda p(h_B) 
          > (q_B(h_A) - \epsilon) \delta(B).
\end{align*}
Taking $\epsilon \rightarrow 0$ we conclude that $q_B(h_A) \delta(B) \leq p(h_A)$ for all $h_A \in \sH$.  Recall that $q_B(h_B) \delta(B)= p(h_B)$.

For $h_A,h_B \in \sH$ we write $h_A \preceq h_B$ if $h_A(x) \leq h_B(x)$ for all $x$, which, as $A$ and $B$ are convex bodies, holds exactly when $A \subseteq B$. 
The set $\sH$ is a convex cone which, together with the partial order $\preceq$, satisfies the many axioms of a {\em pre-ordered cone} (see \cite{FuchssteinerLusky81convex} for details). For each finite $B$ we have now satisfied the conditions for Theorem 1.2.5 of \cite{FuchssteinerLusky81convex}:
\begin{quotation} \em
    \noindent Let $F$ be a pre-ordered cone and let $p:F \rightarrow \overline{\Re}$ be monotone and sublinear, $q:F \rightarrow \overline{\Re}$ superlinear with $q \leq p$. Then there is a monotone linear $\mu:F \rightarrow \overline{\Re}$ with $q \leq \mu \leq p$.
\end{quotation}
In our example $F$ is the cone $\sH$ of support functions of finite sets. Let $q(h)=\delta(B) q_B(h)$. Let $\mu_B \colon \sH \rightarrow \Re$ be the linear map given by the theorem. It is monotone, linear, and 
\[
q(h) = \delta(B) q_B(h) \leq \mu_B(h) \leq p(h).
\]
Since by definition $p(h_{\{a\}})=\delta(\{a\})=0$ for all $a \in \Re^k$, we have $\mu_B(h_{\{a\}})=0$ for all $a \in \Re^k$. 

Now define $\delta_B$ by $\delta_B(A)=\mu_B(h_A)$ for all finite $A$. Then $\delta_B$ vanishes on singletons, it is monotone, linear, and hence also sublinear. By Proposition~\ref{prop:sublinearProperties}, $(\Re^k,\delta_B)$ is a linear semidiversity.

Because $\delta(B) q_B(h_B)=p(h_B)$, we have that 
\[\delta_B(B)=\mu_B(h_B) = \delta(B) q_B(h_B)=\delta(B)\]
and for general finite $A$ we have 
\[\delta_B(A) =\mu_B(h_A) \leq p(h_A) = \delta(A).\]
Repeating this process for all finite $B \subseteq \Re^k$ we obtain a set of linear semidiversities $\{\delta_B\}_{\text{finite } B \subseteq \Re^k}$ such that $\delta_B \leq \delta$ and $\delta_B(B) = \delta(B)$ for all finite $B  \subseteq \Re^k$. So for all finite $A \subseteq \Re^k$,
\[\delta(A) = \sup\{\delta_B(A) : \mbox{ finite } B \subseteq \Re^k \} = \max\{\delta_B(A) : \mbox{ finite } B \subseteq \Re^k \}, \] 
since the supremum is actually attained when $B=A$.
\end{proof}

\section{Embedding into linear and sublinear diversities} \label{sec:embedding}

We now turn our attention from linear and sublinear semidiversities to characterizations  of when finite semidiversities can be  
embedded into linear or sublinear diversities. Questions about embedding of metric spaces have, of course, been central to metric geometry and its applications, particularly after Linial et al.\ \cite{LinialLondonEtal95} demonstrated the link between approximate embeddings and combinatorial optimization algorithms on graphs. We showed in \cite{BryantTupper14} that an analogous link holds between approximate embeddings of diversities and combinatorial optimization algorithms on hypergraphs. This has been used by \cite{jozefiak2023diversity} to obtain the best current approximation bounds for sparsest cut on some classes on hypergraphs. Here we only consider embeddings without distortion, that is, exact rather than approximate embeddings.

A map $f:X_1 \to X_2$ between two semidiversities $(X_1,\delta_1)$ and $(X_2,\delta_2)$ is an {\em embedding} if $\delta_2(f(A)) = \delta_1(A)$ for all finite $A \subseteq X_1$. Since $f$ is not required to be injective, a semidiversity can be embedded into a diversity. We say that a finite semidiversity $(X,\delta)$ is {\em linear-embeddable} if there is an embedding from $(X,\delta)$ to a linear diversity on $\Re^k$ for some $k$ and {\em sublinear-embeddable} if there is an embedding to some sublinear diversity on $\mathbb{R}^k$, for some $k$. At this stage we allow the dimension $k$ to be arbitrary. 

Theorem~\ref{thm:linearEmbed} gives a characterization of linear-embeddability while Theorem~\ref{thm:sublinearEmbed} gives a characterization of sublinear-embeddability. Minkowski diversities and diversities of negative type were reviewed earlier.  

Before we state our characterization of linear-embeddable semidiversities we need to revisit diversities of negative type, extending the results of \cite{WuBryantEtal19} to semidiversities.
Recall that a finite semidiversity has negative type if it satisfies $\sum_{A,B} x_A x_B\, \delta(A \cup B) \leq 0$ for all vectors $x$ indexed by the subsets of $X$ with $x_\emptyset = 0$ and $\sum_A x_A = 0$. 

\begin{proposition} \label{prop:negtype_properties}
Let $(X,\delta)$ be a finite semidiversity. Then $(X,\delta)$ has negative type if and only if there is an embedding of $(X,\delta)$ into a finite diversity of negative type.
\end{proposition}
\begin{proof}
Let $\phi:X \rightarrow Y$ be an embedding of $(X,\delta)$ into a negative type diversity $(Y,\td)$. 
Let $x$ be a vector indexed by subsets of $X$ such that $x_\emptyset = 0$ and $\sum_{A \subseteq X} x_A = 0$. Let $\tilde{x}$ be a vector indexed by subsets of $Y$ with $\tilde{x}_I = \sum_{A:\phi(A) = I} x_A$ for $I \subseteq \phi(X)$ and  $\tilde{x}_I = 0$ for $I \not \subseteq \phi(X)$. Then $\sum_{I \subseteq X} \tilde{x}_I = \sum_{A \subseteq X} x_A = 0$ and $\tilde{x}_\emptyset = 0$, and
    \begin{align*}
        \sum_{A,B \subseteq X} x_A x_B \delta(A \cup B) & = \sum_{I,J \subseteq Y} \sum_{A:\phi(A) = I} \sum_{B:\phi(B) = J} x_A x_B 
        \, \td(\phi(A \cup B)) \\
        & = \sum_{I,J \subseteq Y} \left(\sum_{A : \phi(A) = I} x_A \right) \left(\sum_{B : \phi(B) = J} x_B \right) \td(I \cup J) \\
        & =  \sum_{I,J \subseteq X} \tilde{x}_I \tilde{x}_J \td(I \cup J)  \\
        & \leq 0,
    \end{align*}
    since  $(Y,\td)$ has negative type. Hence $(X,\delta)$ has negative type.

  For the converse, suppose that $(X,\delta)$ is a semidiversity of negative type. Define the equivalence relation $\sim$ on $X$ by $a \sim b \Leftrightarrow \delta(\{a,b\})=0$. Let $Y$ be a subset of $X$ intersecting each equivalence class of $\sim$ in exactly one point and let $\phi:X \rightarrow Y$ be the map with $\phi(a) = b \Leftrightarrow b \in Y \mbox{ and }a \sim b$. Let $\td$ be the restriction of $\delta$ to $Y$. We claim that $(Y,\td)$ has negative type and that $\phi$ is an embedding. 
 
Let $y$ be a zero-sum vector, indexed by subsets of $Y$, such that $y_\emptyset = 0$. Define a vector $x$, indexed by subsets of $X$, where
\[x_A = \begin{cases} y_A & A \subseteq Y \\ 0 & \mbox{otherwise.} \end{cases}\]
Then $x$ has zero sum, $x_\emptyset = 0$ and 
\[\sum_{A,B \subseteq Y} y_A y_B \td(A \cup B) = \sum_{A,B \subseteq X} x_A x_B \delta(A \cup B) \leq 0.\]
Hence $(Y,\td)$ is a diversity with negative type.

Suppose that $A \subseteq X$. Then 
\[\delta(A) \leq \delta(A \cup \phi(A)) \leq \delta(\phi(A)) + \sum_{a \in A} \delta(\{a,\phi(a)\}) = \delta(\phi(A)) = \td(\phi(A))\]
and 
\[\td(\phi(A)) = \delta(\phi(A)) \leq \delta(\phi(A) \cup A) \leq \delta(A) + \sum_{a \in A} \delta(\{a,\phi(a)\}) = \delta(A),\]
in both cases using that fact that $\phi(a) =\tilde a$. Hence $\phi$ is an embedding from the semidiversity $(X,\delta)$ to a diversity $(Y,\td)$ of negative type.
\end{proof}

\begin{theorem} \label{thm:linearEmbed}
Let $(X,\delta)$ be a finite semidiversity. The following are equivalent:
\begin{enumerate}
\item[(i)]$(X,\delta)$  is linear-embeddable.
\item[(ii)] $(X,\delta)$ is of negative type.
\item[(iii)] $(X,\delta)$ can be embedded into a Minkowski diversity $(\Re^k,\delta_K)$ with kernel equal to a  simplex $K \subseteq \Re^k$.
\end{enumerate}
\end{theorem}

\begin{proof}
(i)$\Rightarrow$(ii). Without loss of generality, let $X \subset \Re^k$ where $(\Re^k,\delta)$ is a linear semidiversity. \cite[Sec. 3]{WuBryantEtal19} define the diversity of negative type $(\Re^\ell,\deltaneg)$ where 
\[
\deltaneg (A) = \sum_{i=1}^\ell \max_{a \in A} a_i - \min_{a \in A} \left\{ \sum_{i=1}^\ell a_i \right\}.
\]
We will determine an embedding $\phi \colon X \rightarrow \Re^\ell$ such that
\[
\delta(A) =  \sum_{i=1}^\ell \max_{a \in A} \phi(a)_i \ \ \ \  \text{and} \ \ \ \  \ \  \sum_{i=1}^\ell \phi_i(x)=0,
\]
thereby showing that $(X,\delta)$ is embeddable in a diversity of negative type, and is therefore of negative type by Proposition~\ref{prop:negtype_properties}.

By Theorem~\ref{thm:linear_characterize}, $\delta$ is given by
    \[
    \delta(A) = \int_{\mathbb{S}^{k-1}} h_A(u) \, \mathrm{d}\nu(u)
    \]
    for a positive measure $\nu$ on the sphere satisfying $\int_{\mathbb{S}^{k-1}} u \, \mathrm{d} \nu(u) = 0$. 
    Every $u \in \mathbb{S}^{k-1}$ induces an ordering $\sigma \colon  \{1,\ldots,n\} \rightarrow X$ of the points in $X$ so that
    \begin{equation} \label{eq:sigma_ordering}
    u^T x_{\sigma(1)} \leq \cdots \leq u^T x_{\sigma(n)}.
    \end{equation}
    We denote the order for a given $u$ by $\sigma_u$ (breaking ties using the lexicographic order). Let the set of all $u$ such that $\sigma_u =\sigma$ be denoted $S_\sigma$.
If we integrate \eqref{eq:sigma_ordering} over $S_\sigma$ with respect to $\nu$ we obtain
  \begin{equation} \label{eq:sigma_ordering_integrated}
    u_\sigma^T x_{\sigma(1)} \leq \cdots \leq u_\sigma^T x_{\sigma(n)}
    \end{equation}
where $u_\sigma= \int_{S_\sigma} u \, d\nu(u)$. So for a given set $A$, and for all $u \in S_\sigma$, the $a \in A$ that maximizes $u^T a$ is the same as the $a$ that maximizes $u_\sigma^T a$. Let us denote this value $a$ by $m_{A,\sigma}$.
    
%For a given ordering $\sigma$, and a given subset $A$ of $X$, let $A_{\text{max},\sigma}$ be the value of $a \in A$ that maximizes $u^T a$ for all $a \in A$ for all $u \in S_\sigma$ and hence also maximizes $u_\sigma^T a$ over all $a \in A$.
%$u^T a$ for  $a \in A$ that maximizes $u^T a$ for . 

    For any set $A \subseteq X$ we have
    \[
    \delta(A) =   \int_{\mathbb{S}^{k-1}} h_A(u) d\nu(u) = \sum_{\sigma} \int_{S_\sigma} h_A(u) d\nu(u) =  \sum_{\sigma} \int_{S_\sigma}  u^T m_{A,\sigma} \,  d\nu(u) = \sum_{\sigma}   u_\sigma^T m_{A,\sigma} = \sum_\sigma \max_{a \in A} u_\sigma^T a.
    \]
   
We define a map $\phi \colon X \rightarrow \Re^{\Sigma}$ where $\Sigma$ is the set of all permutations of members of $X$. We define
\(
\phi(a)_\sigma = u_\sigma^T a 
\)
for $a \in X$. So 
\(
\delta (A) = \sum_{\sigma} \max_{a \in A} \phi(a)_\sigma
\)

Now we observe that for all $a \in X$
\[
\sum_{\sigma} \phi(a)_\sigma =  \sum_{\sigma}  u_\sigma^T a 
 =   \sum_{\sigma} \int_{S_\sigma} u^T a d\nu(u) 
 =   \left(  \int_{\mathbb{S}^{k-1}} u d \nu \right)^T a = 0^T a = 0.
\]
So $(X,\delta)$ is embeddable in $(\Re^{|\Sigma|},\deltaneg)$ which is a  diversity of negative type, and therefore $(X,\delta)$ is a semidiversity of negative type.

(ii) $\Rightarrow $(iii). By Theorem~5.2 of \cite{bryant2023diversities}, any diversity of negative type can be embedded into a Minkowski diversity with a simplex kernel. By Proposition~\ref{prop:negtype_properties} this also holds for semidiversities of negative type.

(iii) $\Rightarrow$ (i) By Theorem~5.2 of \cite{bryant2023diversities}, a Minkowski diversity with simplex kernel  is linear, so $(X,\delta)$ is linear-embeddable. 
\end{proof}

\begin{theorem} \label{thm:sublinearEmbed}
Let $(X,\delta)$ be a finite semidiversity. The following are equivalent:
\begin{enumerate}
\item[(i)]$(X,\delta)$  is sublinear-embeddable.
\item[(ii)] There is a finite collection $\{(X,\delta_\gamma):\gamma \in \Gamma\}$ of semidiversities of negative type such that 
\[\delta(A) = \max\{\delta_\gamma: \gamma \in \Gamma\}\]
for all $A \subseteq X$.
\item[(iii)] $(X,\delta)$ can be embedded into a Minkowski diversity.
%(X,\delta)$ is the maximum of a collection of negative type diversities.
\end{enumerate}
\end{theorem}

\begin{proof}
(i)$\Rightarrow$(ii) 
We may assume $X \subseteq \Re^k$ where $(\Re^k,\delta)$ is a sublinear diversity. By Theorem~\ref{thm:sublinearSup} there is a family of linear semidiversities $\delta_\gamma$ for $\gamma \in \Gamma$ such that $\delta(A) = \max \delta_\gamma(A)$. 
%Since $X$ is finite, it has a finite number of subsets, and so we may assume that $\Gamma$ is finite. 
By Theorem \ref{thm:linearEmbed}, each of $(X,\delta_\gamma)$ is of negative type, and therefore $(X,\delta)$ is the maximum of a collection of negative type  semidiversities. For each subset $B$ of $X$ there is a $\gamma_B$ such that $\delta_{\gamma_B}(B)=\delta(B)$ and $\delta_{\gamma}(B) \leq \delta(B)$ for all other $\gamma$. Hence $\delta(A) = \max_{B \subseteq X} \delta_{\gamma_B} (A)$.
Since $X$ has only finitely many subsets, $\delta$ can be expressed as the maximum of a finite number of $\delta_\gamma$ where $(X,\delta_\gamma)$ are of negative type. \\
(ii)$\Rightarrow$(iii)  Suppose $(X,\delta)$ is the maximum of a finite collection of negative type semidiversities, $\delta_\gamma$ for $\gamma \in \Gamma$. 
By Theorem~\ref{thm:linearEmbed}, each $(X,\delta_\gamma)$ can then be embedded into a Minkowski diversity.  
By Proposition 4.1 (a) in \cite{bryant2023diversities} the maximum of a finite collection of Minkowski embeddable diversities is Minkowski embeddable. Since $\Gamma$ is finite, the the diversity given by 
\[\delta(A) = \max\{\delta_\gamma(A):\gamma \in \Gamma\} \]
can also be embedded in a Minkowski diversity.\\
(iii)$\Rightarrow$(i) If $(X,\delta)$ is embeddable into a Minkowski diversity, then it is sublinear-embeddable, since by  \cite[Thm.~2.4]{bryant2023diversities} all Minkowski diversities are sublinear. 
\end{proof}

\bibliographystyle{abbrv}
\bibliography{diversity}

\end{document}